%% file: Version4modificada.tex
\documentclass{amsart}

\usepackage[english]{babel}
\usepackage{amssymb}
\usepackage[all]{xy}
\usepackage{mathtools}
\usepackage{mathabx}

\DeclareMathOperator{\lcm}{lcm}

\title[Genus Field for an Elementary Abelian Extension]
{Genus Field and Extended Genus Field of an Elementary Abelian Extension of Global Fields.}

\author[J.C. Hernandez-Bocanegra]{Juan Carlos Hernandez--Bocanegra}
\address{Departamento de Control Autom\'atico\\
Centro de Investigaci\'on y de Estudios Avanzados del I.P.N.}
\email{juan.cuencame@gmail.com}

 \author[G. Villa]{Gabriel Villa--Salvador}
\address{Departamento de Control Autom\'atico\\
Centro de Investigaci\'on y de Estudios Avanzados del I.P.N.}
\email{gvillasalvador@gmail.com, gvilla@ctrl.cinvestav.mx}

\subjclass[2010]{Primary 11R58; Secondary 11R29, 11R60}

\keywords{Global fields, genus fields, extended
genus fields, Dirichlet characters, class fields}

\date{August 14., 2022}

\newtheorem{teo}{Theorem}[section]

\newtheorem{pro}[teo]{Proposition}
\theoremstyle{definition}
\newtheorem{definition}[teo]{Definition}

\newtheorem{rem}[teo]{Remark}

\begin{document}

\begin{abstract}
In the present work we give the construction of the genus field and the extended genus field of an elementary abelian $l$-extension of a field of rational functions, where $l$ is a prime number. In the Kummer case, if $K$ is contained in a cyclotomic funtion field, the construction is given using Leopoldt's ideas by means of Dirichlet characters. Following the definition of Anglès and Jaulent of extended Hilbert class field, we obtain the extended genus field of an elementary abelian $l$-extension of a field of rational functions.
\end{abstract}

\maketitle

\section{Introduction}
\hfill\break
\input{introduction}
\section{Preliminaries and notation}
\hfill\break
\input{prelimaries}

\section{Genus field when $K/k$ is a Kummer extension}
\subsection{Genus field when $K$ is a cyclotomic function field }
Let $K/k$ be an elementary abelian $l$-extension. Let $K=K_{1}\cdots K_{m}$ where $K_{j}/k$ is a cyclic extension of degree $l$, $1\leq j\leq m$. When $K/k$ is a Kummer extension, each $ K_ {j} $ can be written as a radical extension as follows: $ K_{j} = k (\sqrt[l] {\gamma_{j} D_{j}}) $ for all $ 1 \leq j \leq m $, where $ \gamma_{j} \in \mathbb{F}_{q}^{*} $ and $ D_{j} \in R_{T} = \mathbb{F}_q[T]$ is an $l$-power free monic polynomial. We obtain from \cite[Corolario 9.5.12]{calderon2016campos}, that $K_{j} = k(\sqrt[l] {\gamma_{j} D_{j}}) \subseteq k(\Lambda_{D_{j}})$ if and only if $ \gamma_{j} \equiv (-1)^{\deg\: D_{j}} \bmod {(\mathbb {F}_{q } ^ {*}) ^ {l}} $ which is equivalent to $ \xi_{j}: = (- 1) ^ {\deg \: D_{j}} \gamma_ {j} \in (\mathbb {F}_{q} ^ {*}) ^ {l} $. There are two subcases, namely, whether or not $K$ is contained in a cyclotomic function field.

First we consider the case when $K$ is contained in a cyclotomic function field.  In this case, $ K_ {j} $ is equal to $ k (\sqrt [l] {(-1)^{\deg\,D_{j} } D_ {j}}) $, for all $ 1\leq j \leq m $. Also note that, if $ l | \deg \: D_ {j} $, then $ k (\sqrt [l] {D_{j}}) \subseteq k (\Lambda_ {D_{j}}) $. We use the notation $ k(\sqrt [l] {(- 1) ^ {\deg \: A} A}) = k (\sqrt [l] {A^{*}}) $ for $A\in R_{T}\setminus\{0\}$.
 
 Let $K\subseteq k(\Lambda_{N})$ for $N\in R_{T}$ and let $\chi_{i}$ be the character associated with $K_{i}$, $i\leq i \leq m $. Then $X\coloneqq \langle\chi_{1},\dots \chi_{m}\rangle$ is the group of Dirichlet characters associated with $K$. Since $\chi_ {i}$ is the character associated with $ K_{i}/k$, which is a cyclic extension of degree $l$, we have $ \mathrm{Gal}(K_{i} / k) \cong \langle \chi_ {i} \rangle$ with $ \circ(\chi_{i}) = l$. Therefore $ X \cong C_ {l}^{m}$. Let $ G:= \mathrm{Gal}(K / k)\cong C_{l}^{m}$.
 
 Let $ S: = \{P \in R_{T}^{+} \; | \; P \big| D_{j}, \; \text{for some} \; 1 \leq j \leq m \}$ be the set of monic irreducible polynomials dividing some $ D_{j} $, $ 1\leq j \leq m $. We have that $ S $ is finite, say $ S = \{P_{1}, \dots, P_{r} \}$. We considered the $ P_{i} $-part of each character $ \chi_{j} $. Let $\langle \chi_{P_{i}}\rangle \coloneqq \langle (\chi_{j})_{P_{i}}\;|\;1\leq j\leq m\rangle$, for all $ 1 \leq i \leq r $. Let $ k (\sqrt[l] {P_{i}^{*}}) $ be the field that corresponds to $ \chi_{P_ {i}} $. We have that $ K \subseteq k (\sqrt[l]{P_ {1}^{*}}) \cdots k (\sqrt [l] {P_ {r} ^ {*}}) \subseteq k (\Lambda_ {N}) $. Let $D_{j}=P_{1}^{\beta_{1j}}\cdots P_{r}^{\beta_{rj}}$ with $1\leq j\leq m$, $0\leq\beta_{ij}\leq l-1$, $1\leq i\leq r$.

For $P\in R_{T}^{+}$, let  $ X_ {P} = \{\chi_{P} \; | \; \chi \in X \} $. Note that the conductor of each $ \chi_{P_{i}} $ is precisely $P_{i}$, $1\leq i\leq r$ (see \cite[Definición 9.4.5]{calderon2016campos}), and $X_{P_{i}} = \langle \chi_{P_{i}} \rangle $. Let
  \begin{equation*}
Y =\prod_{P\in R_{T}^{+}}X_{P}= \prod_{P \in S} X_{P} = \prod_{i = 1}^{r} \langle \chi_{Pi} \rangle.
\end{equation*}
Let $L$ be the field associated with $Y \subseteq \widehat {\mathrm{Gal}(k (\Lambda_{N})/k)}$. We have that $L=k(\sqrt[l]{P_{1}^{*}}, \dots, \sqrt[l]{P_{r}^{*}})$, when $K$ is contained in a cyclotomic field. We also have that $e_{P_{i}}(L | k) = | Y_{P_ {i}} | $. Therefore $ | Y_{P_{i}} | = | X_{P_{i}} |$, that is
  \begin{equation*}
      e_{P_{i}}(L| k)=|Y_{P_{i}}|=|X_{P_{i}}|= e_{P_{i}}(E| k).
  \end{equation*}
Hence $ e_{P_{i}}(L| E)=1$ for all $P_{i}\in S$, where $E=k(\sqrt[l]{D_{1}^{*}},\dots, \sqrt[l]{D_{m}^{*}})$ and no finite prime in $E$ is ramified in $ L/E$. In fact we have $E_{\mathfrak{gex}}=L$ (see \cite[Theorem 2.1]{barreto2018genus}). In this case we have $E=K$.\\
\[
\xymatrix {
           k (\Lambda_{P_{1} \cdots P_{r}}) \ar@{-}[d] & \\
           E_{\mathfrak{gex}} = L\ar@{-}[d] \ar@{<->}[r] & Y \\
           E = K \ar@{-}[d] \ar@{<->}[r] & X \\
           k}
\]
Next, our objective is to find $K_{\mathfrak{ge}}$ when $K$ is contained in a cyclotomic function field. To this end, we split the study in several cases.

Let $P_{1}, \dots, P_{r} $ be arranged in such a way that, $ l \big| \deg \: P_{i} $ for $ 1 \leq i \leq s $ and $ l \nmid \deg \: P_{j} $ for $ s + 1 \leq j \leq r $. If $ s <r $, choose $ a, b \in \mathbb{Z} $ such that $ al + b \deg \: P_{r} = 1 $. Let $  d_ {i}=\deg \:P_ {i} $ for $ 1 \leq i \leq r $ and define $ Q_{i}\coloneqq P_ {i} P_ {r} ^ {-bd_{i}} $ for $ 1 \leq i \leq r-1 $. Since
\begin{equation*}
\begin{split}
    \deg\: Q_{i} &= \deg\: P_{i} +(-bd_{i})\deg\: P_{r}= d_{i}+(-bd_{i})d_{r} = d_{i}(1-bd_{r})=d_{i}(al),
\end{split}    
\end{equation*}
$l|\deg Q_{i}$. We define the subfield $M$ of $E_{\mathfrak{gex}}$, as
$M\coloneq k(\sqrt[l]{Q_{1}},\dots,\sqrt[l]{Q_{r-1}})$. We have that $k(\sqrt[l]{Q_{i}})\subseteq E_{\mathfrak{gex}}$, $1\leq i\leq r-1$ and $[M:k]=l^{r-1}$. Thus $[E_{\mathfrak{gex}}:M]=l$. It's easy to see that $E_{\mathfrak{gex}}=M(\sqrt[l]{P^{*}_{r}} )$.

\subsubsection {Case 1: If $ l \big| \deg \: P_ {i} $ for all $ 1 \leq i \leq r $.} 
\hfill \break

From Proposition \ref{pro: 2.5} we obtain that $ \mathfrak {p} _ {\infty} $ is not ramified in $ k (\sqrt[l]{P^{*}_{i}}) / k $, that is, $ e_{\infty} (k (\sqrt [l] {P_ {i} ^ {*}})| k) = 1 $ for all $ 1 \leq i \leq r $. Therefore $\mathfrak{p}_{\infty}$ is not ramified in $ k (\sqrt [l] {P_ {1} ^ {*}}, \dots, \sqrt [l] {P_ {r} ^ {*}}) / k $. It follows that no finite prime is ramified in $ L/ E $ and $ \mathfrak {p} _ {\infty} $ is fully decomposed in $E_{\mathfrak{gex}}/ k $, in particular in $ E_{\mathfrak{gex}} / E $, since in a cyclotomic function field, we have $ f_ {\infty} (E_{\mathfrak{gex}}| k) = 1 $. Thus $E_{\mathfrak{ge}}=E_{\mathfrak{gex}} $.

\subsubsection {Case 2: If $ l \nmid \deg \: D_ {j} $ for some $ 1 \leq j \leq m $.} \label {case: 2}
\hfill \break

From Proposition \ref{pro: 2.5} we have that $ \mathfrak {p} _ {\infty} $ is ramified in $ k (\sqrt [l] {\smash{D_ {j}} ^ {*}}) / k $, and $ e_{\infty} (K_ {j}| k) = l$. From Abhyankar's Lemma follows that
\begin{equation*}
\begin{split}
      e_{\infty} (E| k) & = \lcm [e_{\infty} (K_ {1}| k), \dots, e_{\infty} (K_ {j}| k), \dots, e_{\infty} (K_ {m}| k)] = l. \\
     \end{split}
\end{equation*}
Since $ e_{\infty} (E_{\mathfrak{gex}}| k) \big| l $ we have $ e_{\infty} (E_{\mathfrak{gex}}| E) = 1 $ and therefore $ E_ {\mathfrak{ge}} = E_ {\mathfrak{gex}} $.

\subsubsection {Case 3: If $ l \big| \deg \: D_ {j} \; \forall \; 1 \leq j \leq m $ and $ s<r $.}
\hfill \break

Under these conditions we have that $ \mathfrak {p} _ {\infty} $ decomposes in $ K_ {j} / k $ for all $ 1 \leq j \leq m $ and therefore, $ \mathfrak {p} _ {\infty} $ decomposes in $ E / k $. Now, since $ l \nmid \deg \: P_ {r} $ we have that $ \mathfrak {p} _ {\infty} $ is ramified in $ k (\sqrt [l] {P_{r}} *) /k$ and therefore in $ E_{\mathfrak{gex}} / k $. We will prove that $E_{\mathfrak{ge}}=M$. 

In fact, to see that  $ E \subseteq M $ it suffices to see that $ E_ {j}\coloneq k(\sqrt[l]{\smash{D_{j}}^{*}}) \subseteq M $ for all $ 1 \leq j \leq m $. We have that  $k (\sqrt[l]{\smash{D_{j}}^{*}}) = k (\sqrt[l]{D_{j}}) $ for all $ 1 \leq j\leq m $. Now
\begin{equation}\label{ecua:4.2}
\begin{split}
   Q_{1}^{{\beta_{1{j}}}}\cdots Q_{r-1}^{{\beta_{(r-1){j}}}}&= (P_{1}P_{r}^{{-bd_{1}}})^{{\beta_{1{j}}}}\cdots(P_{r-1}P_{r}^{{-bd_{r-1}}})^{{\beta_{(r-1){j}}}}\\
    &= P_{1}^{\beta_{1{j}}} P_{r}^{-bd_{1}\beta_{1{j}}} \cdots P_{r-1}^{\beta_{(r-1){j}}} P_{r}^{-bd_{r-1}\beta_{(r-1){j}}}\\
    &= P_{1}^{\beta_{1{j}}}\cdots P_{r-1}^{\beta_{(r-1){j}}} \cdot P_{r}^{-b(\beta_{1{j}}d_{1}+\cdots +\beta_{(r-1){j}}d_{r-1})}. 
\end{split}    
\end{equation}
Since $l\big|\deg\,D_{j}$, there exist $q_{j}\in \mathbb{Z}$ such that $-b(\beta_{1{j}}d_{1}+\cdots +\beta_{(r-1){j}}d_{r-1}  )=-b(lq_{j}-\beta_{r{j}}d_{r})$. Since $al+bd_{r}=1$, it follows 
\begin{align*}
    -blq_{j}+\beta_{r_{j}}bd_{r}&= -blq_{j}+\beta_{r_{j}}(1-al)=-blq_{j}+\beta_{r_{j}}-\beta_{r_{j}}al\\
    &= \beta_{r_{j}}+ l(-bq_{j}-\beta_{r_{j}}a)= \beta_{r_{j}}+l\delta.
\end{align*}   

Hence, from equation (\ref{ecua:4.2}) we obtain that
\begin{equation*}
\begin{split}
   Q_{1}^{{\beta_{1{j}}}}\cdots Q_{r-1}^{{\beta_{(r-1){j}}}}
   &=P_{1}^{\beta_{1{j}}}\cdots P_{r-1}^{\beta_{(r-1){j}}} \cdot P_{r}^{\beta_{r{j}}}P_{r}^{l\delta}.
\end{split}    
\end{equation*}
Thus
\begin{equation*}
\begin{split}
   Q_{1}^{\frac{\beta_{1{j}}}{l}}\cdots Q_{r-1}^{\frac{\beta_{(r-1){j}}}{l}}&=P_{1}^{\frac{\beta_{1{j}}}{l}}\cdots P_{r-1}^{\frac{\beta_{(r-1){j}}}{l}} \cdot P_{r}^{\frac{\beta_{r{j}}}{l}}P_{r}^{\delta}=\sqrt[l]{D_{j}}\cdot P_{r}^{\delta},
  \end{split}
\end{equation*}
with $ P_{r}^{\delta} \in k $. Therefore $ E_{j} = k(\sqrt[l] {D_{j}}) \in M $ for all $ 1 \leq j \leq m $. Hence $ E = E_{1} \cdots E_{m} \subseteq M $. 
 
Now, it only remains to prove that the finite primes are not ramified in $M/E$ and that infinite primes over $\mathfrak{p}_{\infty}$ of $E$ are totally decomposed in $M/E$. We have the tower of fields $ k \subseteq E \subseteq M \subseteq E_{\mathfrak{gex}} $. Due to $ e_{P_{i}} (E_{\mathfrak{gex}}| E) = 1 $, it follows that $ e_{P_{i}} (M| E) = 1 $. Since $l\big|\deg Q_{i}$ for all $1\leq i\leq r-1$  we have that $S_{\infty}(E)$, the primes over $\mathfrak{p}_{\infty}$ in $E$ are totally decomposed. Given that $[E_{\mathfrak{gex}}:M]=l$ it follows that $M/E$ is the maximum unramified extension at the finite primes, and totally decomposed at the infinite primes. Therefore we conclude $E_{\mathfrak{ge}}=M$.  

\begin{rem}
We have that $E_{\mathfrak{ge}}=M$ holds if and only if $ l \big| \deg \: D_{j} $ for all $ 1 \leq j \leq m $ and $ l \nmid \deg\, P_{r} $. Note that there exists $1\leq j\leq m$ such that $P_{r}\big| D_{j}$, with $ \deg \: D_{j} = \beta_{1{j}} d_{1} + \cdots + \beta_{r{j}} d_{r} $, that is $\beta_{rj}\neq 0$. Since $ l \big| \deg \: D_{j} $ and $l\nmid d_{r}$, there exists $1\leq i\leq r-1$ such that $l\nmid \beta_{ij}d_{i}$, and so, $l\nmid d_{i}$. It follows $l\nmid d_{r-1}$.
Thus, there exist at least two $ P_{i}'$s such that their degrees are not divisible by $ l $.
\end{rem}

\subsection{Genus field when $ K $ is not a cyclotomic function field}
\hfill\break
Now we study the non-cyclotomic case, that is, $ K \nsubseteq k (\Lambda_{N}) $ for all $N\in R_{T}$. Since $\mathfrak{p}_{\infty}$ is tamely ramified in the abelian extension $K/k$, from Proposition \ref{pro:2.3} we have that there exist $ N \in R_{ T} $ and $ m_{0} \in \mathbb {N} $ such that $ K \subseteq k (\Lambda_{N}) \mathbb {F}_{q^{m_{0}}} $.

Recall that 
\begin{equation*}
\begin{split}
       E & = K_{m_{0}} \cap k (\Lambda_{N}) = (K_{1} \cdots K_{m} \mathbb{F}_{q ^{m_{0}}}(T )) \cap k (\Lambda_{N}).
\end{split}
\end{equation*}
Thus, we have $ E = E_{1} \cdots E_{m} $ where, $ E_{j} = K_{j} \mathbb {F}_{q^{m_{0}}} \cap k (\Lambda_{N}) $, $ 1 \leq j \leq m $. From the Galois correspondence we obtain that $ E_{j} = k (\sqrt[l] {\smash{D_{j}}^{*}}) $ for all $ 1 \leq j \leq m $. Therefore 
\begin{equation*}
     E = k (\sqrt [l]{D_{1}^{*}}, \dots, \sqrt [l] {D_{m}^{*}}).
\end{equation*}

We also have $ \gamma_{j} \not \equiv (-1)^{\deg\:D_{j}} \bmod{(\mathbb{F}_{q}^{*})^{l}}$ for some $j$. Let $ \xi_{j} =(-1)^{\deg \: D_{j}} \gamma_{j} $. We obtain
\begin{equation*}
\begin{split}
    EK &=E(\sqrt[l]{\xi_{1}},\dots, \sqrt[l]{\xi_{r}} )  
       = K(\sqrt[l]{\xi_{1}},\dots, \sqrt[l]{\xi_{r}} )
       =E\mathbb{F}_{q^{l}}
       =K\mathbb{F}_{q^{l}}. 
\end{split}
 \end{equation*}

Next, we analyze the behavior of $\mathfrak{p}_{\infty}$ in the extension $K/k$, when $K\neq E$. Let $G=\mathrm{Gal}(K/k)\cong C_{l}^{m}$ we have $f_{\infty}(K|k)= \frac{|D_{\infty}(K:k)|}{|I_{\infty}(K:k)|}$,
the cardinality of the Galois group of the residue fields. 

Since the residue fields are finite fields we have that $D_{\infty}(K:k)/I_{\infty}(K:k)$ is a cyclic subgroup of $G=\mathrm{Gal}(K/k)=C_{l}^{m}$. The only cyclic subgroups of $G$ are $\langle \mathrm{Id}\rangle $ and $C_{l}$. It follows that $f_{\infty}(K|k)\big|l$. From Abhyankar's Lemma we also have $ e_{\infty} (K| k) \big| l $. 

The number of subgroups of $G$ of order $l$ is $\frac{l^{m}-1}{l-1}$.
Since $G$ is abelian, $G$ has $\frac{l^{m}-1}{l-1}$ subgroups of index $l$, or equivalently, of order $l^{m-1}$. Let $B_{t}$, $1\leq t\leq \frac{l^{m}-1}{l-1}$, these subgroups. For each $t\in\ldbrack 1,\frac{l^{m}-1}{l-1}\rdbrack$, let $F_{t}$ be the fixed field by $B_{t}$
\[
\xymatrix{
 K \ar@{-}[d]_{l^{m-1}} \\
F_{t}=K^{B_{t}}\ar@{-}[d]_{l}\\
 k
}
\]
It follows, from Galois theory, that there exist $ \frac {l^{m} -1} {l-1} $ subfields $F_{t}$ of $ K / k $ such that $ [F_{t}: k] = l $.

\begin{pro}\label{pro:3.2}
The fields $F_{t}$, $1\leq t\leq \frac{l^{m}-1}{l-1}$ are precisely the fields given by $F\footnotesize_{{\overrightarrow{\alpha_{t}}}}\coloneqq k(\mu_{1}^{\alpha_{1t}}\cdots\mu_{m}^{\alpha_{mt}})$ where $\mu_{j}\coloneqq \sqrt[l]{\gamma_{j}D_{j}}$, $0\leq \alpha_{it}\leq l-1$, $1\leq j\leq m$ and $\overrightarrow{\alpha_{t}}=(\alpha_{1t},\dots,\alpha_{mt})\neq (0,\dots,0)$.
\end{pro}
\begin{rem}
We note that, the number of fields described in Proposition \ref{pro:3.2} is $l^{m}-1$. We have that $F\footnotesize_{{\overrightarrow{\alpha_{t}}}}=F\footnotesize_{{\overrightarrow{\alpha_{t'}}}}$ if only if $\overrightarrow{\alpha_{t}}=\lambda\overrightarrow{\alpha_{t'}}\bmod{l}$, with $0\leq \lambda\leq l-1$. Therefore we obtain $\frac{l^{m}-1}{l-1}$ different fields, each $F\footnotesize_{{\overrightarrow{\alpha_{t}}}}\subseteq K$ and $[F\footnotesize_{{\overrightarrow{\alpha_{t}}}}:k]=l$.
\end{rem}
We will use the following notation for $F_{t}\coloneq F_{\footnotesize\overrightarrow{\alpha_{t}}}=k(\sqrt[l]{\eta_{t}R_{t}})$, where
\begin{equation}\label{ecua:3}
 \eta_{t}=\gamma_{1}^{\alpha_{1t}}\cdots \gamma_{m}^{\alpha_{mt}};\;\; R_{t}=D_{1}^{\alpha_{1t}}\cdots D_{m}^{\alpha_{mt}};\hspace{0.2cm} \overrightarrow{\alpha_{t}}=(\alpha_{1t},\dots,\alpha_{mt})\in (\mathbb{F}_{l}^{m})^{*}   
\end{equation}

We have that $ e_{\infty} (K| k) \big| l $. If $ e_{\infty} (K| k) = l $, by Abhyankar's Lemma, there exists $1\leq j\leq m$ such that $e_{\infty}(K_{j}|k)=l$.
\\

We will see the analogous result for the inertial degree of $ \mathfrak{p}_{\infty} $. That is:
 \begin{equation}\label{ecua:4.5}
     f_{\infty}(K|k)=l \Longleftrightarrow \exists\; t\in\Big \ldbrack 1,\frac{l^{m}-1}{l-1}\Big\rdbrack\; \;\text{such that}\;\; f_{\infty}(F_{t}|k)=l. 
 \end{equation} 
 
If $ f_{\infty} (F_{t}| k) = l $ for some $t$, then $f_{\infty}(K|k)=l$.
We prove the sufficiency by induction on $ m $. For $ m = 1 $ we have $\frac{l^{m}-1}{l-1}=1$ and there is nothing to prove.

Suppose the result holds for $m-1$, $m\geq 2$. We have $e_{\infty}(K|k)|l$. If $e_{\infty}(K|k)=1$ then $D_{\infty}(K:k)\cong C_{l}$, since $I_{\infty}(K:k)=\langle\mathrm{Id}\rangle $. 
\\
Let $\mathcal{H}_{1}<G$ be such that $|\mathcal{H}_{1}|=l$ and $\mathcal{H}_{1}\neq D_{\infty}(K:k)$. Let $K'=K^{\mathcal{H}_{1}}$. Since, $ \mathcal{H}_{1} \neq D_{\infty} (K: k) $ we have $ f_{\infty} (K| K')=1 $ and $ f_{\infty} (K'| k) = l $. Since $ K'$ is a field generated by $ m-1 $ subfields of degree $ l $ over $ k $, by the induction hypothesis we obtain that there exists $ t \in \ldbrack 1 , \frac {l^{m-1} -1} {l-1} \rdbrack $ with $ [F_{t}: k] = l $ such that $ f_{\infty} (F_{t}| k) = l $.

If $ e_{\infty} (K| k) = l $ we have that $ D_{\infty} (K: k) \cong C_{l} \times C_{l} $ since $ I_{\infty} (K: k) \cong C_ {l} $. We have
\begin{equation*}
    \frac{D_{\infty}(K:k)}{I_{\infty}(K:k)} \cong C_{l}.
\end{equation*}
Let $J\cong \mathrm{Gal}\Big(\mathcal{O}_{\mathfrak{P}_{\infty}}/\mathfrak{P}_{\infty}:{\mathcal{O}_{\mathfrak{p}_{\infty}}/\mathfrak{p}_{\infty}}\Big).$ Choose $ \mathcal{H}_{2}<G $ such that $ | \mathcal{H}_{2} | = l $ and $ \mathcal{H}_{2} \neq J $. If $ K'' = K^{\mathcal{H}_{2}} $, then $ f_{\infty} (K|K '') = 1 $ and $ f_{\infty} (K''| k) = l $. We have that $ K'' $ is generated by $ m-1 $ subfields of degree $ l $ over $ k $. We have the following diagram
 \[
 \xymatrix{
 & K\ar@{-}[ld]_{_{f_{\infty}(K|K'')=1}} \ar@{-}[rd]^{_{f_{\infty}(K|K^{J})=l}}   &\\
 K'' \ar@{-}[rd]_{_{f_{\infty}(K''|k)=l}} & & K^{J} \ar@{-}[ld]^{_{f_{\infty}(K^{J}|k)=1}}\\
 & k &
 }
\]
By induction, there exists $ t \in \ldbrack 1, \frac {l^{m-1} -1} {l-1} \rdbrack $ with $ k \subseteq F_{t} \subseteq K'' \subseteq K $ such that $[F_{t}: k] = l $ and $ f_{\infty} (F_{t}| k) = l $. In particular, $ F_{t} $ belongs to the lattice of subfields of degree $ l $ over $ k $ of the extension $ K / k $.\\

Now, from \cite[Teorema 9.1.7]{calderon2016campos} we have:
\begin{equation*}
    f_{\infty}(\mathbb{F}_{q^{l}}(T)|k)=\frac{l}{\gcd(d_{k}(\mathfrak{p}_\infty),l)}=l.
\end{equation*}
Since $K\mathbb{F}_{q^{l}}(T)/k$ is an elementary abelian $l$-extension, the inertia group is a cyclic group of order $1$ or $l$ and
$ f_{\infty} (K \mathbb{F}_{q^{l}}(T)| k) = l $. Consider the following tower of fields: $ k\subseteq K\subseteq EK = K\mathbb {F}_{q^{l}} $. Thus
\begin{equation}\label{ecua:4.10}
    f_{\infty}(EK|K)=l \Longleftrightarrow f_{\infty}(K|k)=1.
\end{equation} 
Let $H'\coloneqq D_{\infty}(EK:K)$ so $f_{\infty}(EK|K)=|H'|$. From (\ref{ecua:4.10}) it follows that
\begin{equation*}
    |H'|\neq 1 \Longleftrightarrow f_{\infty}(K|k)=1.
\end{equation*}

Since the extensions $ E / k $ and $ EK / K $ are of degrees $ l^{m} $ and $ l $ respectively, from the Galois correspondence, $ K $ corresponds to $ E \cap K $ in the extension $ E / k $ and therefore $ [E \cap K: k] = l^{m-1} $.

To compute the genus field of $ K $ we need to study the behavior of $ \mathfrak{p}_{\infty} $ in $K/k$. So, if $ f_{\infty} (EK| K) = l $, since $ f_{\infty} (E| k) = 1 $, we have that $ H|_{E} = I_{\infty} (E: k) $. In particular $ e_{\infty} (E| k) = l $. From Case 2, $ E_{\mathfrak{ge}}=E_{\mathfrak{gex}} = k (\sqrt[l]{P_{1}^{*}}, \dots, \sqrt[l]{P_{r}^{*}}) $. Hence, if $ f_{\infty} (EK| K) = l $ then $ H'\cong C_{l} $. Thus, $[E_{\mathfrak{ge}}: E _{\mathfrak{ge}}^{H_{1}}] = l$ where $H_{1}=H'|_{E_{\mathfrak{ge}}}$, it follows that that $E\nsubseteq M$.

Now we see that $ M = E_{\mathfrak{ge}}^{H_{1}} $. For this it is enough to prove that $ M \subseteq E_{\mathfrak{ge}}^{H_{1}} $. Recall that $ Q_{i} = P_{i} P_{r}^{-bd_{i}} $, where $ d_{i} = \deg \: P_{i} $ and that $ l |\deg \: Q_{i} $ for all $ 1 \leq i \leq r-1 $. Thus $ \mathfrak {p}_{\infty} $ decomposes in $ M / k $. Now, $ H=D_{\infty} (EK: K) $ and $H|_{E}= I_{\infty} (E: k) $ we have that $ \mathfrak{p}_{\infty} $ decomposes fully in $ E_{\mathfrak{ge}}^{H_{1}} / k $. Therefore we have $ M \subseteq E_{\mathfrak{ge}}^{H_{1}} $ and 
\begin{equation}\label{ecua:4.11}
  E_{\mathfrak{ge}}^{H_{1}}=E_{\mathfrak{gex}}^{H_{1}}=M.
\end{equation}

To show that $m_{0}=l$, first we recall the following result. 
\begin{teo}\label{teo:3.4}
Let $K$ be a finite abelian extension of $k$. Let $n,m_{0}\in \mathbb{N}$ and $N\in R_{T}$ be such that $K\subseteq {}_{n}k(\Lambda_{N})_{m_{0}}$ and such that $m_{0}$ is minimum with this property. Then $m_{0}$ is independent of $n$ and $N$. Let $t=f_{\infty}(K|k)$ be the degree of the infinite primes of $K$. Let $\mathcal{M}=L_{n}k_{m_{0}}$, $E=K\mathcal{M}\cap k(\Lambda_{N})$, $ \mathcal{F}=K\cap {}_{n}k(\Lambda_{N})$ and $d=f_{\infty}(EK|K)=f_{\infty}(E_{\mathfrak{gex}}K|K_{\mathfrak{ge}})$. Then
\begin{equation*}
    {}_{n}k(\Lambda_{N})K={}_{n}k(\Lambda_{N})_{m_{0}}
\end{equation*}
and
\begin{equation*}
    m_{0}=[K:F]=te_{\infty}(K|F)=tdp^{s}=f_{\infty}(EK|k)p^{s}
\end{equation*}
for some $s\geq 0$. In particular
\begin{equation*}
    e_{\infty}(K|F)=dp^{s}=f_{\infty}(K\mathbb{F}_{q^{m_{0}}}|K).
\end{equation*}
\end{teo}
\begin{proof}
See \cite[Theorem 5.1]{barreto2018genus}.
\end{proof}
\begin{rem}
In particular when $K/k$ is tamely ramified at $\mathfrak{p}_{\infty}$, we have $s=0$ and $m_{0}=td$. In the general case, we may have $s\geq 1$.
\end{rem}
Now, we have that $K/k$ is tamely ramified at $\mathfrak{p}_{\infty}$. From Theorem \ref{teo:3.4} and equation (\ref{ecua:4.10}), we have $m_{0}=f_{\infty}(EK|K)f_{\infty}(K|k)=l$.\\

Consider the following diagram
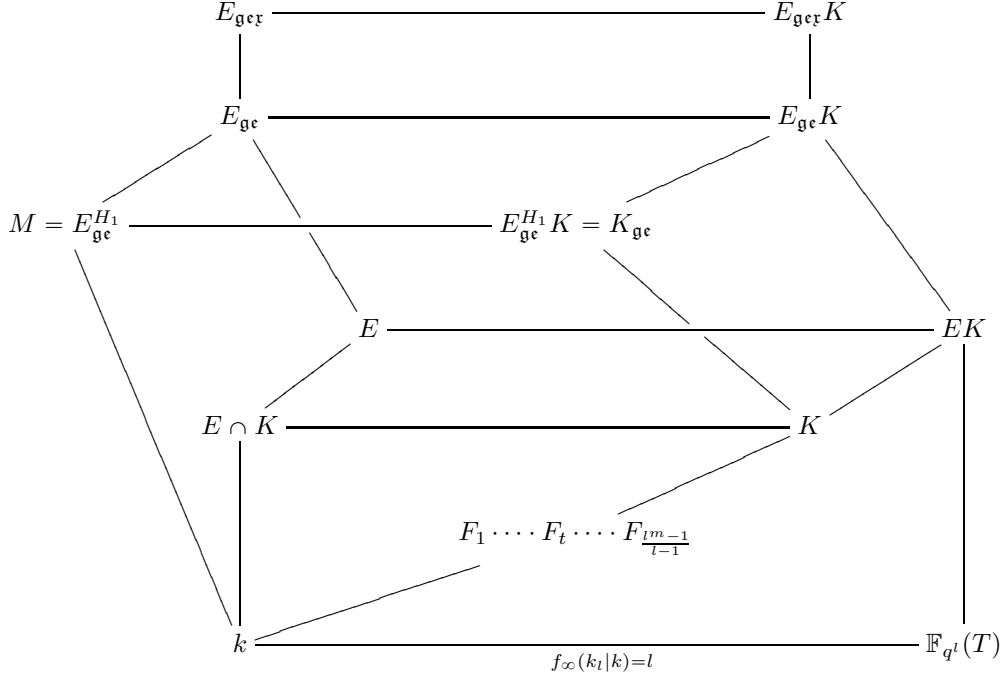
\[
\xymatrix{
           &E_{\mathfrak{gex}}\ar@{-}[rrr] &&&   E_{\mathfrak{gex}}K\\
           &E_{\mathfrak{ge}}\ar@{-}[u] \ar@{-}[rrr] &&&   E_{\mathfrak{ge}}K \ar@{-}[u] \\
          M=E_{\mathfrak{ge}}^{H_{1}} \ar@{-}[ru]\ar@{-}[rrr] & & &  E_{\mathfrak{ge}}^{H_{1}}K=K_{\mathfrak{ge}} \ar@{-}[ru] \ar@{-}[] & \\
          && E\ar@{-}[rrr] \ar@{-}[luu]|-{\hole}  &&& EK \ar@{-}[luu]\\
          &E\cap K\ar@{-}[ru] \ar@{-}[rrr] &&& K\ar@{-}[ru] \ar@{-}[luu]|-{\hole}\\
          &&&F_{1}\cdots\cdot F_{t}\cdots\cdot F_{\frac{l^{m}-1}{l-1}}\ar@{-}[ru]&&\\
          &k \ar@{-}[rrrr]_{f_{\infty}(k_{l}|k)=l} \ar@{-}[luuuu] \ar@{-}[rru]\ar@{-}[uu]\ar@{-}[]&&&& \mathbb{F}_{q^{l}}(T) \ar@{-}[uuu]
}
\]
Next we describe all cases of the genus field of $ K = k (\sqrt[l]{\gamma_{1}D_{1}}) \cdots k(\sqrt[l]{\gamma_{m}D_{m}}) $ over $k$. 

\subsubsection {Case 4: If $ l | \deg \: P_{i} \; \text {for all} \; 1 \leq i \leq r $.}
\hfill\break
In Case 1 it was proved that the genus field of $ E $ is $ E_{\mathfrak{ge}} = E_{\mathfrak{gex}} $. Then $ f_{\infty} (EK| K) \big| e_{\infty} (E| E \cap K) = 1 $. It follows that $ H = D_{\infty} (E_{\mathfrak{ge}} K: K) = \langle \mathrm {Id} \rangle $. It follows that
\begin{equation*}
    K_{\mathfrak{ge}}=E_{\mathfrak{ge}}^{H_{1}}K=E_{\mathfrak{gex}}K.
\end{equation*}

\subsubsection{Case 5: If $ l | \deg\,D_{j}  \; \text{for all} \; 1 \leq j \leq m \; \text{and} \; s <r $.}
\hfill\break
In Case 3 it was proved that the genus field of $ E $ is $ E_{\mathfrak{ge}} = M $. Again
$ f_{\infty} (EK| K) \big| e_{\infty} (E| E \cap K) = 1 $ so that
\begin{equation*}
    K_{\mathfrak{ge}}=E_{\mathfrak{ge}}^{H_{1}}K= MK.
\end{equation*}

\subsubsection {Case 6: If $ l \nmid \deg\,D_{j} \; \text{for some} \; 1 \leq j \leq m $  and $\mathfrak{p}_{\infty}$ is inert in some extension $F_{t}/k$.}
\hfill\break
In Case 2 we obtained that the genus field of $ E $ is $ E_{\mathfrak{ge}} = E_{\mathfrak{gex}}  $. We have that there are extensions $F_{t}/k$ with $F_{t}=k(\sqrt[l]{\eta_{t}R_{t}})$, such that, $f_{\infty}(F_{t}|k)=l$. From Proposition \ref{pro: 2.5}, we have that $\eta_{t}\notin (\mathbb{F}_{q}^{*})^{l}$ and that $l\big|\deg R_{t}$.

Therefore, from (\ref{ecua:4.5}) and (\ref{ecua:4.10}) it follows that $ f_{\infty}(EK| K) = 1 $, or equivalently, $ H = \langle \mathrm {Id} \rangle $. Hence, the genus field of $ K $ is equal to
\begin{equation*}
    K_{\mathfrak{ge}}=E_{\mathfrak{ge}}^{H_{1}}K=E_{\mathfrak{gex}}K.
\end{equation*}

\subsubsection {Case 7: If $ l \nmid \deg\,D_{j} \; \text{for some} \; 1 \leq j \leq m $ and $f_{\infty}(K|k)=1$.}
\hfill\break

It was obtained in Case 2 that $ E_{\mathfrak{ge}}=E_{\mathfrak{gex}} $.
Since $f_{\infty}(K|k)=1$, we have that $f_{\infty}(EK| K ) = l $. Therefore $ H_{1} \cong D_{\infty}(EK: K)\cong C_{l} $. From (\ref{ecua:4.11}) follows that, the genus field of $ K $ is
\begin{equation*}
    K_{\mathfrak{ge}}=E_{\mathfrak{ge}}^{H_{1}}K=E_{\mathfrak{gex}}^{H_{1}}K=MK.
\end{equation*}. 

\begin{subsection}{Extended Genus Field.}
 Now, we compute the extended genus field $K_{\mathfrak{gex}}$ of $K$.

For Cases 1 to 3 we obtain $E_{\mathfrak{gex}}=k(\sqrt[l]{P_{1}^{*}}\dots, \sqrt[l]{P_{r}^{*}})$, and since $K=E$, from Theorem \ref{teo:6.2} follows that $K_{\mathfrak{gex}}=E_{\mathfrak{gex}}=E_{\mathfrak{gex}}K$. For Cases 4 to 6 we have that $H=\langle \textrm{Id}\rangle$, from Theorem \ref{teo:6.3} follows that $K_{\mathfrak{gex}}=E_{\mathfrak{gex}}K$.
 
In Case 7, we obtained $H\cong C_{l}$. We have that $E_{\mathfrak{ge}}^{H}=k(\sqrt[l]{Q_{1}},\dots, \sqrt[l]{Q_{r-1}})=M$ and $E_{\mathfrak{gex}}=E_{\mathfrak{ge}}=E_{\mathfrak{ge}}^{H}(\sqrt[l]{P_{r}^{*}})$. It follows that  
$[E_{\mathfrak{gex}}:E_{\mathfrak{ge}}^{H}]=l$, that $l\nmid \deg P_{r}$ and that the infinite primes are totally ramified in $E_{\mathfrak{ge}}/E_{\mathfrak{ge}}^{H}$. Note that we have in this case two possibilities, $l|d_{r-1}$ or $l\nmid d_{r-1}$. 

From \cite[Teorema 10.3.1]{calderon2016campos} we have that the ramification index of $P_{r}$ in $E_{\mathfrak{ge}}^{H}/k$ is $e_{P_{r}}(E_{\mathfrak{ge}}^{H}|k)=\frac{l}{c}$, where $c=\gcd(-bd_{r-1},l)$. 

When $l\nmid d_{r-1}$ we have that $P_{r}$ is ramified in $E_{\mathfrak{ge}}^{H}/k$. Now, since $(E_{\mathfrak{ge}}^{H})_{\mathfrak{gex}}\subseteq (E_{\mathfrak{ge}})_{\mathfrak{gex}}=E_{\mathfrak{ge}}$ and $E_{\mathfrak{ge}}/k$ is unramified at the finite primes, it follows that $(E_{\mathfrak{ge}}^{H})_{\mathfrak{gex}}=E_{\mathfrak{ge}}$. From Theorem \ref{teo:6.3} we conclude that $K_{\mathfrak{gex}}=E_{\mathfrak{ge}}K=E_{\mathfrak{gex}}K$.

When $l|d_{r-1}$, we have that $P_{r}$ is ramified in $E_{\mathfrak{ge}}/(E_{\mathfrak{ge}}^{H})_{\mathfrak{gex}}$. Therefore $[E_{\mathfrak{ge}}:(E_{\mathfrak{ge}}^{H})_{\mathfrak{gex}}]=l$. Hence $(E_{\mathfrak{ge}}^{H})_{\mathfrak{gex}}=E_{\mathfrak{ge}}^{H}$.

Note that $K_{\mathfrak{gex}}$ may be $E_{\mathfrak{gex}}K$ or $E_{\mathfrak{ge}}^{H}K$. 
We have
\begin{align*}
f_{\infty}=f_{\infty}(K|k)&=[\mathcal{O}_{\mathfrak{P}_{\infty}}/\mathfrak{P}_{\infty}:\mathcal{O}_{\mathfrak{p}_{\infty}}/\mathfrak{p}_{\infty}]=[\mathbb{F}_{q^{f_{\infty}}}:\mathbb{F}_{q}].
    \end{align*}
For Case 7, we have that $f_{\infty}(K|k)=1$, so the degree of $\mathfrak{p}_{\infty}$ in $K$ is $1$. Therefore, from Theorem \ref{teo:4.1.3}, it follows that the constant field of $K_{\mathfrak{ge}}$ is $\mathbb{F}_{q}(T)$.     
We have that $\mathbb{F}_{q^{l}}(T)\subseteq EK\subseteq E_{\mathfrak{gex}}K$. Therefore $E_{\mathfrak{gex}}K/K_{\mathfrak{ge}}$, is an extension of constants, that is $E_{\mathfrak{gex}}K=K_{\mathfrak{ge}}\mathbb{F}_{q^{l}}$, where the constant field of $E_{\mathfrak{gex}}K$ is $\mathbb{F}_{q^{l}}$.  

With the conditions of Case 7, we have that $\gamma_{j}\not\equiv (-1)^{n_{j}}\bmod{(\mathbb{F}_{q}^{*})^{l}}$ for some $1\leq j\leq m$. Note that $\mathfrak{p}_{\infty}$ is ramified in $k(\sqrt[l]{\gamma_{j}D_{j}})/k$ since $f_{\infty}(K|k)=1$, so $l\nmid n_{j}$. Since $(k(\sqrt[l]{\gamma_{j}D_{j}}))_{\mathfrak{gex}}\subseteq K_{\mathfrak{gex}}$, it is enough to find the constant field of the extended Hilbert class field $K_{j_{H^{+}}}$ of $K_{j}=k(\sqrt[l]{\gamma_{j}D_{j}})$.

We have that $e_{\infty}(K_{j}|k)=l$ and $f_{\infty}(K_{j}|k)=1$. Let $K_{j_{\infty}}=k_{\infty}(\sqrt[l]{\gamma_{j}D_{j}})$ be the completion of $K_{j}$ at $\mathfrak{p}_{\infty}$. Fix $\pi_{\infty}=\frac{1}{T}$ a prime element in $k_{\infty}$. Let
\begin{equation*}
\begin{split}
     D_{j}&=T^{n_{j}}+ a_{n_{j}-1}T^{n_{j}-1}+\cdots a_{1}T +a_{0}\\
    &= T^{n_{j}}\Big(1+ a_{n_{j}-1}\big(\frac{1}{T}\big)+\cdots a_{1}\big(\frac{1}{T}\big)^{n_{j}-1}+a_{0}\big(\frac{1}{T}\big)^{n_{j}}\Big)
\end{split}
 \end{equation*}
and define $u\coloneq (1+ a_{n_{j}-1}(\frac{1}{T})+\cdots a_{1}(\frac{1}{T})^{n_{j}-1}+a_{0}(\frac{1}{T})^{n_{j}})\in U_{\infty}^{(1)} $. Since $l\nmid p$, $(U_{\infty}^{(1)})^{l}=U_{\infty}^{(1)}$. Therefore, there exists $v\in U_{\infty}^{(1)} $ such that $u=v^{l}$. Let $n_{j}=lm_{j}-r_{j}$, with $0< r_{j}<l$ and $m_{j},r_{j}\in \mathbb{Z}$. Hence
\begin{equation*}
    K_{j_{\infty}}=k_{\infty}\Big(\sqrt[l]{\gamma_{j}T^{-r_{j}}(T^{m_{j}}v)^{l}}\Big)=k_{\infty}\Big(\sqrt[l]{\frac{\gamma_{j}}{T^{r_{j}}}} \Big)=k_{\infty}\Big(\sqrt[l]{\frac{\gamma_{j}^{s_{j}}}{T}} \Big)
\end{equation*}
with $s_{j}r_{j}\equiv 1\bmod{l}$. Now $(\gamma_{j}/T^{r_{j}})^{s_{j}}=(\gamma_{j}^{s_{j}}/T^{r_{j}s_{j}})=(\gamma_{j}^{s_{j}}/T)(1/T)^{la}$. Therefore 
\begin{equation*}
    K_{j_{\infty}}=k_{\infty}\Big(\sqrt[l]{\frac{\delta}{T}}\Big),\hspace{0.2cm} \text{with} \hspace{0.2cm} \delta=\gamma_{j}^{s_{j}}\in \mathbb{F}_{q}^{*}.
\end{equation*}
Note that $\delta\not\equiv (-1)\bmod{(\mathbb{F}_{q}^{*})^{l}}$. Otherwise, if $\delta\equiv (-1)\bmod{(\mathbb{F}_{q})^{l}}$, since $s_{j}r_{j}=1+la$, we have $\delta^{r_{j}}=\gamma_{j}^{r_{j}s_{j}}=\gamma_{j}^{(1+la)}=\gamma_{j}\gamma_{j}^{la}$, and $\delta^{r_{j}}\equiv (-1)^{r_{j}}\equiv (-1)^{lm_{j}}(-1)^{-n_{j}}\bmod{(\mathbb{F}_{q}^{*})^{l}}$. Thus
\begin{equation*}
    (-1)^{n_{j}}\gamma_{j}\equiv \Big( \frac{(-1)^{m_{j}}}{\gamma_{j}^{a}}\Big )^{l}\in (\mathbb{F}_{q}^{*})^{l}.
\end{equation*}
It follows that $\delta\not\equiv(-1)\bmod{(\mathbb{F}_{q}^{*})^{l}}$. 

Now, we define $\Pi\coloneq\pi_{K_{j_{\infty}}}=\sqrt[l]{\delta/T}$. Whence $\Pi^{l}=\delta\pi_{\infty}$. We have
\begin{equation*}
\textrm{Irr}(\Pi,x,k_{\infty})=x^{l}-\delta\pi_{\infty}=\prod_{j=0}^{l-1}(x-\zeta_{l}^{j}\Pi_{K_{j_{\infty}}}).    
\end{equation*}
Therefore $(-1)^{l}\prod_{j=0}^{l-1}(\zeta_{l}^{j}\Pi)=-\delta\pi_{\infty}$. Thus $N_{K_{j_{\infty}}|k_{\infty}}\Pi=(-1)^{l-1}\delta\pi_{\infty}$. Now, let $\Pi u$ be any other prime from $K_{j_{\infty}}$, with $u\in U_{K_{j_{\infty}}}$. We have $u=\xi v$, where $\xi\in \mathbb{F}_{q}^{*}$ and $v\in U_{K_{j_{\infty}}}^{(1)}$. Hence
\begin{equation*}
    N_{K_{j_{\infty}}|k_{\infty}}(\Pi u)=(-1)^{l-1}\delta\pi_{\infty}\xi^{l}w, \hspace{0.3cm} \text{with}\hspace{0.3cm} w=N_{K_{j_{\infty}}|k_{\infty}}v.
\end{equation*}
Hence $\phi_{K_{j_{\infty}}}(\Pi u)=(-1)^{l-1}\delta\xi^{l}$. Now, let $x\in K_{j_{\infty}}^{*}$ be  an arbitrary element, $x=\Pi^{b}\xi v$. Therefore $ N_{K_{j_{\infty}}|k_{\infty}}(x)=((-1)^{l-1}\delta\pi_{\infty})^{b}\xi^{l}w$, so that \begin{equation*}
\begin{split}
 \phi_{K_{j_{\infty}}}(x)&=\phi_{\infty}(((-1)^{l-1}\delta\pi_{\infty})^{b}\xi^{l}w)= (-1)^{(l-1)b}\delta^{b}\xi^{l}.
\end{split}
\end{equation*}
Now, we want to find the minimum $b\geq 1$ and $\xi\in \mathbb{F}_{q}^{*}$ such that $x\in\ker \phi_{K_{j_{\infty}}}$, that is $(-1)^{(l-1)b}\delta^{b}\xi^{l}=1$. Note that, if we assume $b=1$, then
$(-1)^{l-1}\delta=\xi^{-l}\in (\mathbb{F}_{q}^{*})^{l}$. Thus $(-1)^{l}\delta\equiv (-1)\bmod{(\mathbb{F}_{q}^{*})^{l}}\equiv \delta \bmod{(\mathbb{F}_{q}^{*})^{l}}$, and $\delta\equiv (-1)\bmod{(\mathbb{F}_{q}^{*})^{l}}$. This is a contradiction. Now, since $\delta^{-1}\Pi_{K_{j_{\infty}}}^{l}\in \ker \phi_{K_{j_{\infty}}}$,  we have that 
\begin{equation*}
  \min \{m\in \mathbb{N}\;|\; \text{there exists}\; \overrightarrow{\alpha}\in U_{K_{j}}^{+}\;\text{with}\; \deg\overrightarrow{\alpha}=m \}=l.   
\end{equation*}
It follows that the field of constants of $K_{j_{H^{+}}}$ is $\mathbb{F}_{q^{l}}$ (Theorem \ref{teo:2.10}). 

Therefore the field of constants of $K_{\mathfrak{gex}}$ is $\mathbb{F}_{q^{l}}$, and the field of constants of $K_{\mathfrak{ge}}$ is $\mathbb{F}_{q}$. Finally, we have that $K_{\mathfrak{gex}}=E_{\mathfrak{gex}}K$.
\end{subsection}

We summarize the Kummer case.

\begin{teo}\label{teo:4.1.4}
Let $ K / k $ be a geometric extension of congruence functions fields, such that $K=K_{1}\cdots K_{m}$ with $K_{j}=k(\sqrt[l]{\gamma_{j}D_{j}})$, $D_{j}$ a monic $l$-power free polynomial, $\gamma_{j}\in \mathbb{F}_{q}^{*}$, $1\leq j\leq m$ and  $\mathrm{Gal}(K/k)=C_{l}^{m}$. Let $\xi_{j}\coloneqq (-1)^{\mathrm{gr}\: D_{j}}\gamma_{j}$ and $\{P_{1},\dots,P_{r}\in R_{T}^{+}\;|\; P|D_{j}\quad \text{for some} \quad j\leq t\leq m \}$. Set $n_{j}=\deg\:D_{j}$ and $d_{i}=\deg\,\:P_{i}$ for $1\leq j\leq m$ and $1\leq i\leq r$. We arrange the polynomials $ P_{i}'$s such that $l|d_{i}$ $1\leq i\leq s$ and $l\nmid d_{j}$, $s+1\leq j\leq r$, $0\leq s\leq r$. Let $E=k(\sqrt[l]{D_{1}^{*}},\dots,\sqrt[l]{D_{m}^{*}})$. Then

\begin{equation*}
    E_{\mathfrak{gex}}= \left\{ \begin{array}{lcc}k(\sqrt[l]{P_{1}},\dots,\sqrt[l]{P_{r}})&  \text{if} &s=r \\
    k(\sqrt[l]{P_{1}},\dots, \sqrt[l]{P_{s}}, \sqrt[l]{\smash{P_{s+1}}^{*}},\dots \sqrt[l]{\smash{P_{r}}^{*}} )& \text{if}& s\leq r-1.
    \end{array}
   \right.
  \end{equation*}  
Let $ Q_{i} \coloneqq P_{i}P_{r}^{-bd_{i}} $, $ 1 \leq i \leq r-1 $ with $al+bd_{r}=1$ and $M=k(\sqrt[l]{Q_{1}},\dots,\sqrt[l]{Q_{r-1}})$. Then the genus field $ K_{\mathfrak{ge}} $ of $ K $ over $k$ is:
\\

\item [(a)] $K_{\mathfrak{ge}}=E_{\mathfrak{gex}}K$ if any of the following conditions holds: 
\begin{itemize}
 
    \item $l|d_{r}$ ($s=r$). 
    
     \item $\xi_{j}\in  (\mathbb{F}_{q}^{*})^{l}$  $\forall \;1\leq j\leq m$ and $l\nmid n_{j}$ for some $1\leq j\leq m$.
     
     \item $\xi_{j'}\notin (\mathbb{F}_{q}^{*})^{l}$ for some $1\leq j'\leq m$, $l\nmid n_{j}$ for some $1\leq j\leq m$ and $f_{\infty}(K|k)=l$. (See equation (\ref{ecua:4.5}))
   
\end{itemize}
\item[(b)]$K_{\mathfrak{ge}}=MK$ if any of the following conditions holds:
\begin{itemize}
 
     \item  $l| n_{j}$ for all $1\leq j\leq m$.

    \item $\xi_{j'}\notin (\mathbb{F}_{q}^{*})^{l}$ for some $1\leq j'\leq m$, $l\nmid n_{j}$ for some $1\leq j\leq m$ and $f_{\infty}(K|k)=1$. (See equation (\ref{ecua:4.5})).
\end{itemize}
 The extended genus field $K_{\mathfrak{gex}}$ of $K$ over $k$ is:
\begin{equation*}
    K_{\mathfrak{gex}}=E_{\mathfrak{gex}}K.
\end{equation*}

\end{teo}

\section{Genus field when $K/k$ is not a Kummer extension}
\hfill\break

\begin{pro}
Let $ K / k $ be a finite abelian extension of global functions fields, $ P \in R_ {T}^{+} $ with $ d = \deg\: P $. Suppose that $ P $ is tamely ramified in $ K / k $. If $ e = e_{P} (K| k) $ denotes the ramification index of $ P $ in $ K / k $, then $ e | q^{d} -1 $, where the constant field of $ K $ is $ \mathbb{F}_{q}$.
\end{pro}
\begin{proof}
See \cite[Proposición 10.4.8]{calderon2016campos}
\end{proof}
Now we assume that $K/k$ is an elementary abelian extension that is not a Kummer extension.
In our case, we have that the inertia group of $\mathfrak{p}_{\infty}$, is $C_{l}$ or $\langle\mathrm{Id}\rangle$. If $ I_{\infty}(K: k) \cong C_{l} $, then $ e_{\infty} (K| k) = l $ and since $\deg\:\mathfrak{p}_{\infty} = 1 $, we have $ l | q^{1} -1 $, contrary to the hypothesis that $ l \nmid q-1 $. Therefore $ e_{\infty} (K| k) = 1 $, that is, $ \mathfrak{p}_{\infty} $ is not ramified in $ K / k $. We have two cases: $ K $ is contained or not, in a cyclotomic function field.
\subsection{Genus field when $K$ is a cyclotomic function field.}
\hfill\break
When $K$ is a cyclotomic field, we use the notation of Section 3.1. We have that $K=E$ is the field associated with $X$, $L$ is the field associated with $Y$, and $S=\{P_{1},\dots,P_{r}\}$ denotes the set of finite primes that are ramified in $E/k$.

We define $L_{i}\coloneqq k(\Lambda_{P_{i}})^{X_{P_{i}}^{\perp}}$. Then $ L_ {i} $ is the only subfield of $ k (\Lambda_ {P_{i}}) / k $ of degree $ l $ over $ k $, that is, $ k \subseteq L_ {i} \subseteq k (\Lambda_{P_{i}}) $ is such that $ [L_{i}: k] = l $, which is a cyclic extension for $1\leq i\leq r$. Now, we have that $K\subseteq L_{1}\cdots L_{r}\subseteq k(\Lambda_{P_{1}\cdots P_{r}})$. Hence $e_{P_{i}}(L_{1}\cdots L_{r}|k)=|X_{P_{i}}|=e_{P_{i}}(E|k)=l$. Therefore $L_{1}\cdots L_{r}/E$ is an unramified extension. Hence $K_{\mathfrak{ge}}=E_{\mathfrak{ge}}=L_{1}\cdots L_{r}=E_{\mathfrak{gex}}=K_{\mathfrak{gex}}$.

\subsection{Genus field when $K$ is a not cyclotomic field.}
\hfill\break
In this case we have that $EK\neq K $ where $E$ is given by $E\coloneqq K_{m_{0}}\cap k(\Lambda_{N})$. As in the Kummer case, $ EK / K $ is an extension of constants, and therefore unramified. Let $ H '= D_{\infty} (EK: K) $ be the decomposition group of $ \mathfrak{p}_{\infty} $ in $ EK / K $. We have that $ H '|_{E} = I_{\infty} (E: k)=\langle Id\rangle$. From Theorem \ref{teo:4.1.3}, the genus field of $K$ is equal to
\begin{equation*}
    K_{\mathfrak{ge}}=E_{\mathfrak{ge}}^{H_{1}}K=E_{\mathfrak{ge}}K.
\end{equation*}

We summarize the non-Kummer case in the following result.

\begin{teo}
Let $K/k$ be a geometric extension of congruent funtion fields, such that $K=K_{1}\cdots K_{m}$ with $\mathrm{Gal}(K/k)=C_{l}^{m}$ and $l\nmid q-1$. Let $S=\{P\in R_{T}^{+}\;|\;P|D_{T}\;\hspace{0.3cm}\text{for some}\hspace{0.3cm} 1\leq t\leq m\}=\{P_{1},\dots,P_{r}\}$ the set of ramified primes in $K/k$. Then the genus field $K_{\mathfrak{ge}}$ and the extended genus field $K_{\mathfrak{gex}}$ of $K$ over $k$ are:

   \begin{equation*}
        K_{\mathfrak{ge}}=K_{\mathfrak{gex}}=E_{\mathfrak{gex}}K,
    \end{equation*}
    
\noindent where $L=L_{1}\cdots L_{r}$ with $L_{i}$ the subfield of $k(\Lambda_{P_{i}})$ of degree $l$ over $k$, $1\leq i\leq r$.
    
\end{teo}

\end{document}

%% file: introduction.tex
The study of \textit {genus fields} began with C.F. Gauss in \cite{gauss}, in the context of binary quadratic forms. H. Hasse \cite{hasse1951} studied the extended genus field theory for quadratic number field, using characters and class field theory. It was H. W. Leopoldt \cite{leopoldt1953} who determined the extended genus field $ K_{\mathfrak{gex}} $ of an abelian extension $ K $ of $ \mathbb {Q} $, and studied the arithmetic of $K$ using Dirichlet characters, generalizing Hasse's work.

For an arbitrary finite extension $ K/\mathbb {Q} $, the genus field $ K_{\mathfrak{ge}} $ of $ K $ is the composition of $ K $ and the maximum abelian extension $k^{*} $ of $ \mathbb{Q} $ contained in the Hilbert class field of $K$, that is, $ K_{\mathfrak{ge}}=Kk^{*} $. This definition was given by A. Fröhlich \cite{frohlich1983central}. Similarly, the extended genus field $ K_{\mathfrak{gex}} $ of $ K $, is defined as $ K_{\mathfrak{gex}} = KL $ where $ L/\mathbb{Q} $ the maximal abelian extension contained in the extended Hilbert class field of $K$.

For a number field $ K $, the \textit {Hilbert class field} $ K_{H} $ of $ K $ is defined as the maximum abelian extension of $ K $ unramified at any prime. The \textit {extended Hilbert class field} $ K_{H^{+}} $ of $k$ is defined as the maximum abelian extension of $ K $ unramified at every finite prime, these definitions are without any ambiguity. We have that $ K \subseteq K_{\mathfrak{ge}}\subseteq K_{H} $ and $ \textrm{Gal}(K_{H}/K) $ is isomorphic to the class group $ Cl_{K} $ of $ K $. The genus field $ K_{\mathfrak{ge}} $ corresponds to a subgroup $ G_{K} $ of $ Cl_{K} $. We have $ \textrm{Gal}(K_{\mathfrak{ge}}/K) \cong Cl_{K} / G_{K} $. The degree $ [K_{\mathfrak{ge}}: K] $ is called the \textit{genus number} of $ K $ and the group $ \textrm{Gal}(K_{\mathfrak{ge}}/K) $ is known as the \textit{genus group}. Similarly for $ K\subseteq K_{\mathfrak{gex}} \subseteq K_{H^{+}} $.

M. Ishida \cite{ishida2006genus} described the genus field $K_{\mathfrak{ge}} $ for any finite extension $K$ of $ \mathbb{Q} $. X. Zhang \cite{zhang1985} gave a simple expression of $K_{\mathfrak{ge}}$ for any number field, using Hilbert's ramification theory.

In global function fields, there are several definitions of the Hilbert class field depending on the aspect you want to study. Following the ideas of number fields, the definition of the Hilbert class field of a global function field $K$ over $\mathbb{F}_{q}$ as the maximum unramified abelian extension of $ K $, turns out that it contains all constant extensions $ K_{m} \coloneqq K \mathbb {F}_{q^{m}}$ for any positive integer $m$. Therefore, the direct definition of Hilbert class field as the maximum unramified abelian extension, has the disadvantage of being of infinite degree over the base field.

M. Rosen \cite{rosen1987hilbert} gave a definition of the Hilbert class field of $ K $. For a fixed finite non-empty set $ S_{\infty} $ of prime divisors of $ K $, the Hilbert class field of $K$ is defined as the maximum unramified abelian extension of $ K $ such that the primes in $S_{\infty} $ are totally decomposed. Using this definition, a suitable concept of genus field can be given, as in the case of number fields. It was R. Clement \cite{clement1992genus} the first to consider the extended genus field of a cyclic extension $K$ of the field of rational functions $ k =\mathbb{F}_{q}(T) $ of degree a prime number $ l $ dividing $ q -1 $. Using class field theory she found the genus field of $K$. S. Bae and J.K. Koo \cite{bae1996genus} generalized Clement's result using Fröhlich's methods.

G. Peng \cite{peng2003genus} described explicitly the genus field of a cyclic Kummer extension of prime degree of a congruence function field using Rosen's definitions of Hilbert class field. Later on, S. Hu and Y. Li \cite{hu2010genus} described the genus field of an Artin-Schreier extension of a field of rational functions.

In \cite{MALDONADORAMIREZ201340} and \cite{MALDONADORAMIREZ2015283}, the theory of genus fields was developed for a congruence funtion field using Rosen's definition of Hilbert class field. The method used is based on Leopoldt's ideas using Dirichlet's characters when $K$ is contained in a cyclotomic funtion field. If $K$ is not contained in a cyclotomic extension and $\mathfrak{p}_{\infty}$ is tamely ramified, we consider a suitable extension of constants of $K$ and then proceed as before to find $K_{\mathfrak{ge}}$. The method can be used as an application to find explicitly $ K_{\mathfrak{ge}} $ of an extension $K/k$ of prime degree extension that divides $q-1$ (Kummer) or $ l=p $, where $ p $ is the characteristic (Artin- Schreier) and also when $ K / k $ is a $ p $-cyclic extension (Witt). Later on, this method was used in \cite{bautista2013genus} to describe $ K_{\mathfrak{ge}} $ explicitly when $ K / k $ is a cyclic extension of degree $l^{n} $, where $ l $ is a prime number and $ l^{n} \big| q-1$ under several restrictions.
In \cite{maldonado2017genus} the genus field of a finite and separable extension $K$ of $ k=\mathbb{F}_{q} (T) $ is described as $K_{\mathfrak{ge}}=Kk^{*} $ where $k^{*} $ is the composition of two fields $k_{1} $ and $ k_{2} $ following the method of Ishida \cite{ishida2006genus}.

The study of the genus fields of abelian extensions of $k$ has been considered recently in \cite{barreto2018genus}, \cite{bautista2013genus}, \cite{MALDONADORAMIREZ201340}, \cite{MALDONADORAMIREZ2015283}. In \cite{barreto2018genus} the genus field $ K_ {ge} $ of a finite abelian extension $ K $ of $ k =\mathbb{F}_{q}(T) $ is described, in a much more transparent way than in \cite{MALDONADORAMIREZ201340}.

In the present work we characterize the genus field and the extended genus field of an elementary abelian extension $ K $ of degree $ l^{m} $ of a rational congruence function field  $ k =\mathbb{F}_{q}(T) $, where $ K = K_{1} \cdots K_{m} $ and $ \textrm {Gal} (K / k) \cong C_{l}^{m} $, where $ C_{l} $ denotes the cyclic group of order $ l $. Using the ideas obtained in \cite{barreto2018genus}, the genus field is described when $ l \big| (q-1) $ and when $l\nmid (q-1)$. The main result of this work describes the genus field and the extended genus field of $K$ explicitly. The proof of the main result is constructive beginning with the study of the genus field of a cyclic extension $ K / k $ of degree $l$; this is Peng's result \cite{peng2003genus}. We generalize Peng's results to an elementary abelian $l$-extension of degree $l^{m}$.

B. Anglès and J.-F. Jaulent presented in \cite{angles2000theorie} the general theory of the extended genus field of global fields, either function or numeric. E. Ramírez, M. Rzedowski and G. Villa described in \cite{ramirezramirez2019genus} the extended genus field using a different concept from the one defined by Anglès and Jaulent.

In this paper we are interested in describing the extended genus field for an elementary abelian $l$-extension of degree $l^{m}$. We use the class field theory approach of genus theory given by Anglès and Jaulent and its connection with Dirichlet characters given M. Rzedowski and G. Villa in \cite{rzedowskicalderon2021class}, where it is shown that if $K$ is contained in a cyclotomic function field, then $K_{\mathfrak{gex}}$ is the same field for both approaches. Furthermore, in our situation, the elementary abelian $l$-extensions, the extended genus field obtained by means of the Angl\`es-Jaulent approach is the same that the one given by Ramírez, Rzedowski and Villa. For more general extensions the extended genus fields might be a different.

%% file: prelimaries.tex
We use the following notation. Let $k=\mathbb{F}_{q}(T)$ be a rational congruence funtion field, $\mathbb{F}_{q}$ denoting the finite field of $q$ elements. Let $R_{T}=\mathbb{F}_{q}[T]$ be the ring of polynomials, that is, $R_{T}$ is the ring of integers of $k$ and let $R_{T}^{+}\coloneqq \{P\in R_{T}\;|\;P\; \text{is monic and irreducible} \}$. The elements of $R_{T}^{+}$ are the \textit{finite primes} of $k$ and $\mathfrak{p}_{\infty}=\infty$, the pole of the principal divisor $(T)$ in $k$, is the \textit{infinite prime}.  For $N\in R_{T}\setminus\{0\}$, $\Lambda_{N}$ denotes the $N$-torsion of the Carlitz module and $k(\Lambda_{N})$ denotes the $N$-th cyclotomic funtion field. The $R_{T}$-module $\Lambda_{N}$ is cyclic and $\lambda_{N}$ denotes a generator of $\Lambda_{N}$ as $R_{T}$-module.

For a finite Galois extension $L/K$ of function fields, if $\mathfrak{p}$ is a place of $K$ and $\mathfrak{P}$ is a place above $\mathfrak{p}$, the decomposition and the inertia groups are denoted by $D_{\mathfrak{P}/\mathfrak{p}}(L|K)$ and $I_{\mathfrak{P}/\mathfrak{p}}(L|K)$ respectively. For any finite extension $K/k$ we will use the symbol $S_{\infty}(K)$ to denote either one prime or the set of all primes in $K$ above $\mathfrak{p}_{\infty}$. Given a cyclotomic function field, a \textit{Dirichlet character} is a group homomorphism $\chi:(R_{T}/(N))^{*}\longrightarrow \mathbb{C}^{*}$ and we define the conductor $F_{\chi}$ of $\chi$, as the monic polynomial of minimum degree such that $\chi$ can be defined modulo $F_{\chi}$, $\chi:(R_{T}/(F_{\chi}))^{*}\longrightarrow \mathbb{C}^{*}$. 
\theoremstyle{definition}
\begin{definition}
Let $ X $ be any group of Dirichlet characters. Let $ N $ be the least common multiple of $ \{F_{\chi} \; | \; \chi \in X \} $. Thus $ X $ is a subgroup of $ \widehat{G_{N}} $, where $G_{N}=\mathrm{Gal}(k(\Lambda_{N})/k)$. Let $ H = \bigcap_{\chi \in X} \ker \chi $ and $ k_{X}\coloneqq k (\Lambda_{N})^{H} $. Then $ k_{X} $ is called \emph{the field that belongs to $ X $ or field associated with} $ X $.
\end{definition}

For any character $\chi$, let $\chi=\prod_{P\in R_{T}^{+}}\chi_{P}$ be the canonical decomposition. We have that $\chi_{P}$ has conductor a power of $P$. We also have that $F_{\chi}=\prod_{P\in R_{T}^{+}}F_{\chi_{P}}$. If $X$ is a group of Dirichlet characters, we write $X_{P}\coloneqq \{\chi_{P}\;|\;\chi\in X\}$ for $P\in R_{T}^{+}$. If $K$ is any extension of $k$, $k\subseteq K\subseteq k(\Lambda_{N})$ and $P\in R_{T}^{+}$, then the ramification index of $P$ in $K$ is $e_{P}(K|k)=|X_{P}|$.

The ramification index of $\mathfrak{p}_{\infty}$ in $k(\Lambda_{N})/k$ is $e_{\infty}=q-1$ and $\mathfrak{p}_{\infty}$ decomposes in $h_{\infty}=\frac{|G_{N}|}{q-1}$ different prime divisors of $k(\Lambda_{N})$ of degree $1$. We have $G_{N}\cong (R_{T}/(N))^{*}$ and let $\mathfrak{J}$ be the inertia group of $\mathfrak{p}_{\infty}$. Then $\mathfrak{J}=\mathbb{F}_{q}^{*}\subseteq G_{N}$, that is, $\mathfrak{J}=\{\sigma_{a}\:|\: a\in\mathbb{F}_{q}^{*} \}$. Furthermore, the inertia and decomposition groups of $\mathfrak{p}_{\infty}$ coincide. In $k(\Lambda_{N})/k$,  $\mathfrak{p}_{\infty}$ (if $q>2$) and the polynomials $P\in R_{T}^{+}$ such that $P|N$ are the ramified primes.

Let $S_{\infty}$ be any nonempty finite set of prime divisors of $K$. The \textit{Hilbert class funtion field of} $K$ relative to $S_{\infty}$, $K_{H,S_{\infty}}$ is the maximum unramified abelian extension of $K$ where every element of $S_{\infty}$ decomposes fully. This definition is due to Rosen \cite{rosen1987hilbert}. In this paper, we consider $S_{\infty}$ as the set of prime divisors dividing $\mathfrak{p}_{\infty}$, the pole of $T$ in $k$ and we denote $K_{H}=K_{H,S_{\infty}}$.

\begin{definition}
Let $K$ be a finite geometric extension of $k$, that is, the exact field of constants of $K$ is $\mathbb{F}_{q}$. The genus field $K_{\mathfrak{ge}}$ of $K$ is the maximum extension of $K$ contained in $K_{H}$ that is the composite of $K$ and an abelian extension of $k$. Equivalently, $K_{\mathfrak{ge}}=Kk^{*}$ where $k^{*}$ is the maximum abelian extension of $k$ contained in $K_{H}$.
\end{definition}

Let $K/k$ be a finite abelian extension. From the Kronecker-Weber Theorem, we have that there exists $n,m_{0}\in \mathbb{N}$ and $N\in R_{T}$ such that
\begin{equation*}
    K\subseteq {}_{n}k(\Lambda_{N})_{m_{0}}\coloneqq L_{n} k(\Lambda_{N}) \mathbb{F}_{q^{m_{0}}},
\end{equation*}
where $L_{n}$ denotes the subfield of $k(\Lambda_{1/T^{n+1}})$ of degree $q^{n}$ and $k_{m_{0}}=\mathbb{F}_{q^{m}_{0}}(T)$ is the extension of constants of $k$ of degree $m_{0}$. We have that $\mathfrak{p}_{\infty}$ is totally and wildly ramified in $L_{n}/k$. We also have that $\mathfrak{p}_{\infty}$ is totally inert in $k_{m_{0}}/k$.
\begin{pro}\label{pro:2.3}
If $K/k$ is an abelian extension such that $\mathfrak{p}_{\infty}$ is tamely ramified, then there exist $N\in R_{T}$ and $m_{0}\in\mathbb{N} $ such that $K\subseteq k(\Lambda_{N})\mathbb{F}_{q^{m_{0}}}$. 
\end{pro}

\begin{proof}
See \cite[Proposition 3.4]{MALDONADORAMIREZ201340}.
\end{proof}

\begin{definition}
A Kummer extension is an extension $ K / k $ such that $ K = k (\sqrt[n]{\Delta})$, where $ n \in \mathbb{N} $ is relatively prime to the characteristic of $ k $, that $\mu_{n}\subseteq k$ where $\mu_{n}$ is the group of the $ n $-th roots of the unit, $\Delta$ is a subgroup of $k^{*}$ containing $(k^{*})^{n}$, where $k^{*}$ denotes the multiplicative group of the field $k$, and $k(\sqrt[n]{\Delta})$ is the field generated by all the roots $\sqrt[n]{a}$ with $a\in \Delta$.
\end{definition}

\begin{pro} \label{pro: 2.5}
Let $\gamma\in \mathbb{F}_{q}^{*}$ and $D\in R_{T}$ a monic polynomial. The behavior of $ \mathfrak{p}_{\infty} $ in $ k(\sqrt[l]{\gamma D})/k $ is as follows:
\begin{enumerate}
     \item if $ l \nmid \deg \: D $, $ \mathfrak{p}_{\infty} $ is fully ramified.
     \item if $ l | \deg \: D $ and $ \gamma \in (\mathbb{F}_{q}^{*})^{l} $, $ \mathfrak{p}_{\infty} $ is completely decomposed.
     \item if $ l | \deg \: D $ and $ \gamma \notin (\mathbb{F}_{q}^{*})^{l} $, $ \mathfrak{p}_{\infty} $ is totally inert.
    
\end{enumerate}
\end{pro}
\begin{proof}
See \cite[Lemma 3]{peng2003genus}.
\end{proof}

Next result describes the genus field of a finite abelian extension $ K / k $.
\begin{teo}\label{teo:4.1.3}
Let $ K $ be a finite abelian geometric extension of $ k $ where $ \mathfrak{p}_{\infty} $ is tamely ramified and $K\subseteq k(\Lambda_{N})\mathbb{F}_{q^{m_{0}}}$. Let
\begin{equation*}
    E \coloneqq K \mathbb {F}_{q^{m_{0}}}(T) \cap k(\Lambda_{N}).
\end{equation*}
Then
\begin{equation}
    K_{\mathfrak{ge}} = E_{\mathfrak{ge}}^{H_{1}} K =(E_{\mathfrak{ge}} K)^{H}
\end{equation}
where $ H $ is the decomposition group of any prime of $ S_{\infty}(K) $, the set of primes of $ K $ above $ \mathfrak{p}_{\infty} $, in $ E_{\mathfrak{ge}} K / K $, $ H_{1} \coloneqq H|_{E_{\mathfrak{ge}}} $ and $ H_{2} \coloneqq H_{1}|_{E} $. 
Let $ d \coloneqq f_{\infty} (EK| K) $. We have $ H \cong H_{1} \cong H_{2} \cong C_{d} $ and $ d\big| q-1 $. We also have $E_{\mathfrak{ge}}K / K_{\mathfrak{ge}}$ and $EK / E^{H_{2}}K$ are extensions of constants of degree $d$. Finally, the constant field of $ K_{\mathfrak{ge}} $ is $ \mathbb{F}_{q^{t}} $, where $ t $ is the degree of any element of $ S_{\infty} $.
\end{teo}
\begin{proof}
See \cite[Theorem 2.2]{barreto2018genus}.
\end{proof}

\subsection{Extended Genus Field}
\hfill\break
\input{extended_genus_field}

%% file: extended_genus_field.tex
In this section we recall the concept of extended genus field  for a geometric extension of congruence function fields $K/k$ according to \cite{rzedowskicalderon2021class}, where it is obtained the extended genus field of an abelian extension of a rational function field following the definition of Angl{\`e}s and Jaulent \cite{angles2000theorie}. 
 
Let $k=\mathbb{F}_{q}(T)$ be a global rational function field. Let $\mathfrak{p}_{\infty}$ be the infinite prime of $k$, that is, the pole of $T$ and let $k_{\infty}\cong \mathbb{F}_{q}((\frac{1}{T}))$ be the completion of $k$ at $\mathfrak{p}_{\infty}$. For $x\in k^{*}_{\infty}$, $x$ can be written  uniquely as: 
\begin{equation*}
    x=\Big(\frac{1}{T}\Big)^{n_{x}}\lambda_{x}\varepsilon_{x} \hspace{0.3cm} \text{with} \hspace{0.3cm} n_{x}\in \mathbb{Z}, \hspace{0.3cm} \lambda_{x}\in \mathbb{F}_{q} \hspace{0.3cm} \text{and} \hspace{0.3cm} \varepsilon_{x}\in U_{\infty}^{(1)}
\end{equation*}
where  $U_{\infty}^{(1)}=1+\mathfrak{p}_{\infty}$ is the group of 1-th units of $k$.

\begin{definition}
The $\text{sign function}$ $\phi_{\infty}:k^{*}_{\infty}\longrightarrow \mathbb{F}_{q}^{*}$ is defined as $\phi_{\infty}(x)=\lambda_{x} $ for $x\in k_{\infty}^{*}$. The value $\phi_{\infty}(x)$ is called the “sign" of $x$.
\end{definition}
\begin{definition}
Let $L$ be a finite separable extension of $k_{\infty}$. We define the sign of $L^{*}$ by $\phi_{L}\coloneq \phi_{\infty}\circ N_{L/k_{\infty}}:L^{*}\longrightarrow \mathbb{F}_{q}^{*}$, where $N_{L/k_{\infty}}$ is the norm from $L$ to $k_{\infty}$.
\end{definition}
Let $\mathcal{O}_{L}=\{x\in L\;|\;v_{\mathfrak{p}}(x)\geq 0,\; \text{for all}\; \mathfrak{p}\nmid \mathfrak{p}_{\infty} \}$, $I_{L}$ the group of fractional ideals, $P_{L}=\{ (x)=x\mathcal{O}_{L}\;|\; x\in L^{*}\}$ and $C_{L}=I_{L}/P_{L}$. Let $L^{+}=\{x\in L^{*}\;|\; \phi_{L_{\mathfrak{p}}}(x)=1, \;\text{for all}\; \mathfrak{p}\big|\mathfrak{p}_{\infty} \}$ be the set of the totally positive elements, let $P_{L}^{+}=\{x\mathcal{O}_{L}\;|\; x\in L^{+}\}$ and $C_{L}^{+}=I_{L}/P_{L}^{+}$.

\begin{definition}
We define the following subgroups of the group of idèles $J_{L}$:
\begin{equation*}
    U_{L}\coloneq \prod_{v\in \mathcal{P}_{\infty}} L_{v}^{*} \times \prod_{v\notin \mathcal{P}_{\infty}} U_{L_{v}},
   \end{equation*}
   \begin{equation*}
        U_{L}^{+}\coloneq \prod_{v\in \mathcal{P}_{\infty}} \ker\phi_{L_{v}} \times \prod_{v\notin \mathcal{P}_{\infty}} U_{L_{v}}.
   \end{equation*}
\end{definition}
The groups $U_{L}L^{*}$ and $U_{L}^{+}L^{*}$ are open subgroups of $J_{L}$. By class field theory, we have the isomorphism
\begin{equation*}
    C_{L}\cong J_{L}/U_{L}L^{*} \hspace{0.3cm} \text{and} \hspace{0.3cm}  C_{L}^{+}\cong J_{L}/U_{L}^{+}L^{*}.
\end{equation*}
For $S=S_{\infty}(L)$ we denote by $L_{H}=L_{H,S}$ the Hilbert class field of the global function field $L$. By class field theory we have
\begin{equation*}
    \mathrm{Gal}(L_{H}/L)\cong  C_{L}\cong J_{L}/U_{L}L^{*}. 
\end{equation*}

\begin{definition}
Let $L_{H^{+}}=L_{H}^{ext}$ be the abelian extension of the global function field $L$ corresponding to the idèle subgroup $U_{L}^{+}L^{*}$ of $J_{L}$. The field $L_{H^{+}}$ is called the extended Hilbert class field of $L$ corresponding to $\mathcal{O}_{L}$.
\end{definition}
Thus, $L_{H^{+}}/L$ is an unramified extension at the finite prime divisors of $L$, $L_{H}\subseteq L_{H^{+}}$ and
\begin{equation*}
    \mathrm{Gal}(L_{H^{+}}/L)\cong J_{L}/U_{L}^{+}L^{*}\cong C_{L}^{+}.
\end{equation*}
The degree function $\deg :J_{L}\longrightarrow \mathbb{Z}$ is defined as $\deg (\tilde{\alpha})=\sum_{\mathfrak{p}\in \mathbb{P}_{L}}\deg(\mathfrak{p})\cdot v_{\mathfrak{p}}(\alpha_{\mathfrak{p}})$, where $\deg(\mathfrak{p})$ is the degree of $\mathfrak{p}$, $\mathbb{P}_{L}$ denotes the set of prime divisors of $L$. Next result is important for finding the extended genus field.
\begin{teo}\label{teo:2.10}

Let $B<\mathcal{C}_{L}$ be an open subgroup of finite index, where $\mathcal{C}_{L}$ is the idèle class group of $L$. Let $d\coloneq \min\{n\in \mathbb{N}\;|\;\text{there exist}\; \tilde{\mathbf{b}}\in B, \; \text{with}\;\deg \tilde{\mathbf{b}}=n\}$. Then, if $E$ is the field associated with $B$, namely $B=N_{E/L}\mathcal{C}_{E}$, where $\mathcal{C}_{E}$ is the idèle class group of $E$, the field of constants of $E$ is $\mathbb{F}_{q^{d}}$.
\end{teo}
\begin{proof}
See \cite[Teorema 17.8.6]{calderon2016campos}.
\end{proof}
\begin{definition}
Let $L/K$ be a finite separable extension of global funtion fields. We define the extended genus field $L_{\mathfrak{gex},K}$ of $L$ with respect to $K$ as the maximum extension of $L$ contained in $L_{H^{+}}$ that is of the form $LK_{1,+}$ where $K_{1,+}/K$ is an abelian extension. The maximum field $K_{1,+}$ that satisfies  $L_{\mathfrak{gex},K}=LK_{1,K}$ will be denoted by $K_{L,+}$. In this
way, $K_{L,+}$ is the maximum extension of $K$ such that $L_{\mathfrak{gex},K}=LK_{L,+}$.
\end{definition}

Next theorem proves that the natural definition of extended genus field of a cyclotomic funtion field obtained by means of Dirichlet characters is the same as the one given by Angl{\`e}s and Jaulent.

\begin{teo}\label{teo:6.2}
Let $E\subseteq k(\Lambda_{N})$. Then the genus field $E_{\mathfrak{gex}}$ relative to $k$ is the field associated to the group of Dirichlet characters $Y=\prod_{P\in R_{T}} X_{P}$, where $X$ is the group of Dirichlet characters associated with the field $E$.
\end{teo}
\begin{proof}
See \cite[Theorem 4.5]{rzedowskicalderon2021class}.
\end{proof}

The extended genus field of a finite abelian extension $K/k$ was defined in \cite{ramirezramirez2019genus} as $K_{\mathfrak{gex}}=E_{\mathfrak{gex}}K$. According to the definition of Angl{\`e}s and Jaulent of $K_{\mathfrak{gex}}$, we have.

\begin{teo}\label{teo:6.3}
With the notations of Theorem \ref{teo:4.1.3}, we have that $K_{\mathfrak{gex}}=DK$ with $(E_{\mathfrak{ge}}^{H})_{\mathfrak{gex}}\subseteq D\subseteq E_{\mathfrak{gex}}$. In particular, when $H=\langle \mathrm{Id}\rangle$, we have $K_{\mathfrak{gex}}=E_{\mathfrak{gex}}K$.
\end{teo}
\begin{proof}
See \cite[Theorem 5.1]{rzedowskicalderon2021class}.
\end{proof}